\documentclass[a4paper,10pt,reqno]{amsart}        
\usepackage{verbatim}       
\usepackage{amssymb}        
        
\usepackage{enumerate}       
\usepackage[active]{srcltx}       
\numberwithin{equation}{section}      
\usepackage{stmaryrd} 
       
\usepackage{t1enc}       
\usepackage[utf8x]{inputenc}         
 \usepackage{enumerate}

\newtheorem{theorem}{Theorem}[section]        
\newtheorem{lemma}[theorem]{Lemma}

\newtheorem{observation}[theorem]{Observation}

\theoremstyle{definition}       
\newtheorem{definition}[theorem]{Definition}       
\theoremstyle{remark}       
         
\newcommand{\mc}[1]{\mathcal{#1}}       
\newcommand{\mbb}[1]{\mathbb{#1}}

\newcommand{\setm}{\setminus}       
\newcommand{\empt}{\emptyset}       
\newcommand{\subs}{\subset}       
       
\newcommand{\ran}{\operatorname{ran}}

\def\<{\left\langle}       
\def\>{\right\rangle}       

\newcommand{\prop}[1]{\Delta(#1)}

\newcommand{\lsub}[1]{${}_{#1}$}

\newcommand{\acal}{\mc A}

\newcommand{\cf}{cf}

\def\br#1;#2;{\bigl[ {#1} \bigr]^ {#2} }

\newcommand{\MVCTEXT}{\mathbf {MinVC}}

\newcommand{\MWOTEXT}{\mathbf {MaxWO}}

\newcommand{\MMM}[5]{\mathbb(#1,{#2},#3,#4)\to #5}

\newcommand{\MVC}[4]{\mathbf{M}\MMM{#1}{#2}{#3}{#4}{\MVCTEXT}}

\newcommand{\MWO}[4]{\mathbf{M}\MMM{#1}{#2}{#3}{#4}{\MWOTEXT}}

\newcommand{\cad}[2]{\operatorname C(#1,#2)}

\newcommand{\met}[2]{#1[#2] }
\newcommand{\nmet}[2]{#1[- #2]}

\newcommand{\bo}{{\beth_{\omega}}}

\author[T. Csern\'ak]{Tam\'as Csern\'ak}
\address
      {Eötvös University of Budapest, Hungary  }
\email{tamas@csernak.com}

\author[L. Soukup]{Lajos Soukup}
\address
      { Alfr{\'e}d R{\'e}nyi Institute of Mathematics, Eötvös Loránd Research Network, Budapest, Hungary  }
\email{soukup@renyi.hu}

\keywords{infinite hypergraph, vertex cover, minimal vertex cover, Shelah's revised GCH,  almost disjoint, }
\title[Minimal vertex covers]%
   {Minimal vertex covers in infinite hypergraphs}
\thanks{The preparation of this paper was partially
supported by  OTKA grants K 129211}

\date{\today}       
       
\begin{document}       

\begin{abstract}
    If $A$ is a set and $S$ is a set of cardinals,
write
\begin{displaymath}
{[A]}^{S}=\{B\subs A: |B|\in S\}.
\end{displaymath}

In this paper a hypergraph
 will be identified with the family  of its edges.
A hypergraph $\mc E$  possesses  {\em property $\cad{k}{\rho}$} iff
 $|\bigcap \mc E'|<{\rho}$ for each $\mc E'\in {[\mc E]}^{k}$.

If  
${\lambda}$ and ${\rho}$ are  cardinals, $S$ is a set of cardinals, 
$k\in {\omega}$, then we write
\begin{displaymath}
\MVC{{\lambda}}{S}{k}{{\mu}}
\end{displaymath}
iff
every hypergraph  
$\mc E\subs {[{\lambda}]}^{S}$  possessing  property $\cad{k}{{\rho}}$ has a minimal vertex cover. 
If $S=\{\kappa\}$, then we simply write $\mathbf M({{\lambda}},{\kappa},{k},{{\mu}})\to  \mathbf{MinVC}$
for $\mathbf M({{\lambda}},\{{\kappa}\},{k},{{\mu}})\to  \mathbf{MinVC}$. 

A set $S$ of cardinals is {\em nowhere stationary} iff 
    $S\cap {\alpha}$ is not stationary  in ${\alpha}$ for any ordinal ${\alpha}$ with $cf({\alpha})>{\omega}$.
In particular, countable sets of cardinals, and sets of successor cardinals 
 are nowhere stationary.
 
In this paper we prove:
\begin{enumerate}[(1)]
\item $\MVC{{\lambda}}{S}{2}{k}$ for each 
 nowhere stationary set $S$ of cardinals and 
$k<{\omega}\le {\lambda}$ and $k\in {\omega}$,  \smallskip
\item $\MVC{{\lambda}}{{\kappa}}{2}{{\omega}}$ for each 
${\omega}_2\le {\kappa}\le {\lambda}$ provided $GCH$ holds,\smallskip
\item $\MVC{{\lambda}}{{\kappa}}{2}{{\rho}}$ provided 
${\rho}<\beth_{\omega}\le {\kappa}\le {\lambda}$,  \smallskip
\item $\MVC{{\lambda}}{{\omega}}{r}{k}$
provided ${\omega}\le {\lambda}$ and $k,r\in {\omega}$,  \smallskip
\item $\MVC{{\lambda}}{{\omega}_1}{3}{k}$
provided ${\omega}_1\le {\lambda}$ and $k\in {\omega}$.
\end{enumerate}

\end{abstract}

\maketitle

\section{Introduction}

A {\em hypergraph} is a pair of sets $\<V,\mc E\>$ such that 
the elements of $\mc E$ are non-empty subsets of $V$.
We are interested in hypergraphs $\<V,\mc E\>$ whose edges cover $V$ in 
the sense that 
$V = \cup \mc E$. To shorten notations, 
in this case the hypergraph
 will be identified with the family $\mc E$ of its edges. 

Given a hypergraph $\mc E$ 
a vertex set $Y\subs \cup \mc E$ is a {\em vertex cover} of $\mc E$ iff
$Y\cap E\ne \empt$ for each $E\in \mc E$.
We say that a vertex cover  $Z$ is a {\em minimal vertex cover} of $\mc E$
iff  no proper subset of $Z$ is a vertex cover of $
\mc E$.

Using the terminology of Erd\H os and Hajnal \cite{EH} we say that 
a hypergraph $\mc E$  possesses  {\em property $\cad{k}{\rho}$} iff
 $|\bigcap \mc E'|<{\rho}$ for each $\mc E'\in {[\mc E]}^{k}$.

 Dominic van der Zypen  \cite{BaZy,Zy} raised  the following question:
{\em  Assume that $r\in {\omega}$. Does every hypergraphs which possesses property $\cad{2}{r}$ have 
 a minimal vertex cover?}

 In \cite{Ko} Komjáth gave some partial answers. 
 To formulate them precisely we introduce some more notations.

We define the {\em edge cardinality spectrum}
of a hypergraph $\mc E$, $\|\mc E\|$, as the cardinalities of 
the edges of the hypergraph:
\begin{displaymath}
    \|\mc E\|=\{|E|:E\in \mc E\}.
\end{displaymath}

If $A$ is a set and $S$ is a set of cardinals,
write
\begin{displaymath}
{[A]}^{S}=\{B\subs A: |B|\in S\}.
\end{displaymath}
If $S=\{{\kappa}\}$, then we will  use the standard notation ${[A]}^{{\kappa}}$ instead of
${[A]}^{\{{\kappa}\}}$.

If  
${\lambda}$ and ${\rho}$ are  cardinals, $S$ is a set of cardinals, 
$k\in {\omega}$, then we write
\begin{displaymath}
\MVC{{\lambda}}{S}{k}{{\rho}}
\end{displaymath}
iff
every hypergraph  
$\mc E\subs {[{\lambda}]}^{S}$  possessing  property $\cad{k}{{\rho}}$ has a minimal vertex cover.

Komjath (\cite[Corollary 8]{Ko}) proved that 
\begin{enumerate}[(k1)]
    \item  
    $\MVC{{\lambda}}{{\omega}\cup\{{\kappa}\}}{2}{k}$ for each 
    ${\omega}\le {\kappa}\le {\lambda}$.
    \end{enumerate}

In this paper  we prove a strengthening of this result. To formulate our theorem we need to introduce
the following notion. 
A set $S$ of ordinals is {\em nowhere stationary} iff 
    $S\cap {\alpha}$ is not stationary  in ${\alpha}$ for any ordinal ${\alpha}$ with $cf({\alpha})>{\omega}$.
In particular, countable sets of cardinals, and sets of successor cardinals 
 are nowhere stationary.  
 In Section \ref{sc:inhomogeneous} we prove:
 \begin{theorem}\label{tm:nowhere-stac} 
    $\MVC{{\lambda}}{S}{2}{r}$ for each cardinal ${\lambda}$ and $r\in {\omega}$
provided   that $S$ is a nowhere stationary set of cardinals.
 \end{theorem}

Komjáth (\cite[Theorems 16 and 17]{Ko}) also  proved
\begin{enumerate}[(k1)]\addtocounter{enumi}{1}
\item $\MVC{{\lambda}}{{\omega}\cup\{{\kappa}\}}{2}{{\omega}}$ for each 
${\omega}_1\le {\kappa}\le {\lambda}$ provided $V=L$, \smallskip
\item GCH does not imply that 
$\MVC{\aleph_{{\omega}+1}}{{\omega}_1}{2}{{\omega}}$. 
\end{enumerate}

In Section \ref{sc:mwo} we show that the assumption $V=L$ can be relaxed to 
$GCH$ provided  we strengthen the  assumption  ${\kappa}\ge {\omega}_1$ to  
${\kappa}\ge {\omega}_2$.   
    
 \begin{theorem}\label{tm:mvc-gch}
    If GCH holds, then $\MVC{{\lambda}}{{\omega}\cup\{{\kappa}\}}{2}{{\omega}}$
    for each ${\omega}_2\le {\kappa}\le {\lambda}$.
    \end{theorem}

 One can conjecture that one can not obtain positive ZFC theorems 
 for hypergraphs which possesses property $\cad{2}{{\rho}}$ for some ${\rho}\ge {\omega}$. 
 But this is not true: in Section \ref{sc:mwo},   
using Shelah's Revised CGH theorem,   we prove in ZFC: 
    \begin{theorem}\label{tm:mvc-bo}
    $\MVC{{\lambda}}{{\omega}\cup\{{\kappa}\}}{2}{\rho}$
        for each ${\rho}<\beth_{\omega}\le {\kappa}\le {\lambda}$.
    \end{theorem}

In Section \ref{sc:ckr} we consider hypergraphs possessing property $\cad{k}{r}$
for $3\le k<{\omega}$, and we prove the following results.

\begin{theorem}\label{tm:ckr}
    $\MVC{{\lambda}}{{\omega}}{r}{k}$
    provided ${\omega}\le {\lambda}$ and $k,r\in {\omega}$
\end{theorem}

\begin{theorem}\label{tm:c3k}
    $\MVC{{\lambda}}{{\omega}_1}{3}{k}$ 
    provided ${\omega}_1\le {\lambda}$ and $k\in {\omega}$.
\end{theorem}

\subsection*{Notations and basic observations.}

We use the standard notation of infinite combinatorics.  

    If $\mc F\subs \mc P(X)$, $x\in X$ and $Y\subs X$, write 
  \begin{displaymath}
  \mc F(x)=\{F\in \mc F: x\in F\},
  \end{displaymath}
    \begin{displaymath}
    \met{\mc F}{Y}=\{F\in \mc F: F\cap Y\ne \empt\}\text{ and }
    \nmet {\mc F}{Y}=\{F\in \mc F: F\cap Y= \empt\},
    \end{displaymath}
and 
\begin{displaymath}
\mc F\restriction Y=\{F\cap Y: F\in \mc F[Y]\}
\text{ and }
\mc F\sqcap Y=\mc F\cap \mc P(Y).
\end{displaymath}
Observe that $\empt\notin \mc F\restriction Y$.

\begin{observation}
Assume that  $\mc E$ is a hypergraph and  $Y$ is a vertex cover of $\mc E$. 
Then     $Y$ is a minimal vertex cover  of $\mc E$ iff there is a function
    $w:Y\to \mc E$ such that $w(y)\cap Y=\{y\}$ for each $y\in Y$.    
\end{observation}
We will say that  $w$ is a {\em witnessing function}, or that $w$ {\em witnesses  the minimality 
of $Y$}. 

If $\<X_{\alpha}:{\alpha}<{\delta}\>$ is a sequence of sets, and ${\beta}<{\delta}$, write 
$X_{<{\beta}}=\bigcup_{{\alpha}<{\beta}}X_{\alpha}$.
We define $X_{\le {\beta}}$, $X_{>{\beta}}$, etc. analogously.

\section{Maximizing well ordering}\label{sc:mwo}

Given a hypergraph $\mc E$ a well-ordering $\preceq$ of the vertex set of $\mc E$ is called 
a {\em maximizing  well-ordering} iff every edge $E\in \mc E$ has a $\preceq$-maximal element.

If  
${\lambda}$ and ${\rho}$ are  cardinals, $S$ is a set of cardinals, 
$k\in {\omega}$, then 
we write
\begin{displaymath}
\MWO{{\lambda}}{S}{k}{\rho}
\end{displaymath}
iff
every hypergraph  
$\mc E\subs {[{\lambda}]}^{S}$  possessing  property $\cad{k}{{\rho}}$ has a maximizing well-ordering.

In \cite{Ko} Komjáth observed that Klimo \cite[Theorem 5]{Kl} practically  proved that 
 if a hypergraph has a maximizing well order then 
it has a minimal vertex cover, i.e. 
\begin{displaymath}\tag{$\dag$}\label{eq:NWONVC}
    \MWOTEXT\to \MVCTEXT
\end{displaymath}

In \cite[Theorem 7]{Ko} Komjáth actually proved that 
$\MWO{{\lambda}}{{\kappa}}{2}{k}$ for each 
    ${\omega}\le {\kappa}\le {\lambda}$.
So if $\mc E\subs {[{\lambda}]}^{{\omega}\cup \{{\kappa}\}}$, then 
$\mc E\cap {[{\lambda}]}^{{\kappa}}$ has a maximizing well-order $\preceq$,
which is also a maximizing well-order for $\mc E$ because non-empty finite sets
have maximal elements in any ordering.
 
Using the same approach to yield Theorems \ref{tm:mvc-gch} and \ref{tm:mvc-bo} it is enough to prove 
the following two theorems.
\begin{theorem}\label{tm:maxwo-gch}
    If GCH holds, then $\MWO{{\lambda}}{{\kappa}}{2}{{\omega}}$
    for each ${\omega}_2\le {\kappa}\le {\lambda}$.
    \end{theorem}

    \begin{theorem}\label{tm:maxwo-bo}
        $\MWO{{\lambda}}{{\kappa}}{2}{\rho}$
            for each ${\rho}<\beth_{\omega}\le {\kappa}\le {\lambda}$.
        \end{theorem}

First we need a general stepping up method.

\begin{definition}
Let $\mc A$ be a hypergraph, ${\lambda}=|\bigcup A|$.
We say that a continuous, increasing sequence 
$\<G_{\alpha}:{\alpha}<\cf({\lambda})\>\subs {[\bigcup \mc A]}^{<{\lambda}}$ is a {\em good cut} of $\mc A$
iff 
\begin{enumerate}[(i)]
\item 
$G_0=\empt$ and  $\bigcup_{{\alpha}<\cf({\lambda})}G_{\alpha}=\bigcup \mc A$, \smallskip
\item $\forall A\in \mc A$ $\exists {\alpha}<cf({\lambda})$ $A\subs G_{{\alpha}+1}$ and 
$|G_{\alpha}\cap A|<|A|$. 
\end{enumerate} 
\end{definition}

\begin{definition}
Given hypergraphs  $\mc A$ and $\mc B$  we say that $\mc B$ is a {\em shrink}  of $\mc A$ and we 
write 
$\mc B\ll \mc A$ iff $\forall B\in \mc B$ $\exists A\in \mc A$ such that 
$B\subs A$ and $|A\setm B|<|A|$.
\end{definition}

\begin{theorem}\label{tm:gen-step-up}
    Let ${\kappa}$ be an infinite  cardinal. 
    Assume that $\mbb A$ is a collection of hypergraphs
    with the following properties:
    \begin{enumerate}[(a)]
    \item If $\mc A\in \mbb A$ and $\mc B\ll \mc A$, then $\mc B\in \mbb A$.    
 
    \item if $\mc A\in \mbb A$,  and ${\lambda}=|\bigcup\mc A|>{\kappa}$ 
    then $\mc A$ has a good cut.
    \end{enumerate}
  
        \noindent (1) 
        If 
 \begin{enumerate}[(a)]\addtocounter{enumi}{2}
 \item every $\mc A\in \mbb A$ with $|\bigcup\mc A|\le {\kappa}$ has 
 a minimal vertex cover,
 \end{enumerate}       
then every $\mc A\in \mbb A$ has a minimal vertex cover.

\noindent (2) If
\begin{enumerate}[(a')]\addtocounter{enumi}{2}
    \item every $\mc A\in \mbb A$ with $|\bigcup\mc A|\le {\kappa}$ has 
    a maximizing well order, 
    \end{enumerate}       
    then every $\mc A\in \mbb A$ has a maximizing well-order.

\end{theorem}

\begin{proof}
(1) By induction on ${\lambda}=|\bigcup \mc A|$.

If ${\lambda}={\kappa}$, then (c) implies the statement. 

  Assume that ${\lambda}>{\kappa}$.
 Let $\<G_{\alpha}:{\alpha}<\cf({\lambda})\>$ be a good cut of $\mc A$ by (b). 
  
By transfinite induction on ${\alpha}<{\lambda}$ we will define
$Y_{\alpha}\subs G_{{\alpha}+1}\setm G_{\alpha}$
and a function $w_{\alpha}:Y_{\alpha}\to \mc A\sqcap G_{{\alpha}+1}$ 
such that $Y_{\le {\alpha}}(=\bigcup_{{\alpha}'\le {\alpha}}Y_{{\alpha}'})$ is a minimal vertex cover of  $\mc A\sqcap G_{{\alpha}+1}$
witnessed by $w_{\le \alpha}$,
as follows.

Assume that $\<Y_{\zeta}:{\zeta}<{\alpha}\>$ and $\<w_{\zeta}:{\zeta}<{\alpha}\>$
are defined. 
Let \begin{displaymath} 
\mc A_{\alpha}=\big(\mc A[-Y_{<{\alpha}}]\sqcap G_{{\alpha}+1}\big)\restriction 
(G_{{\alpha}+1}\setm G_{\alpha}).
\end{displaymath}
Then $\mc A_{\alpha}\ll \mc A$, so $\mc A_{\alpha}\in \mbb A$ by (a).
Since $|\bigcup\mc A_{\alpha}|\le |G_{{\alpha}+1}|<{\lambda}$,
by the inductive assumption $\mc A_{\alpha}$ has a minimal vertex cover 
$Y_{\alpha}$ witnessed by a function $w_{\alpha}$.

After the inductive construction
put $Y=\bigcup_{{\alpha}<\cf({\lambda})}Y_{\alpha}$
and $w=\bigcup_{{\alpha}<\cf({\lambda})}w_{\alpha}$.
Then $Y$ is a minimal vertex cover of $\mc A$ witnessed by $w$.
\medskip

\noindent (2)
By induction on ${\lambda}=|\bigcup \mc A|$.

If ${\lambda}={\kappa}$, then (c') implies the statement. 

  Assume that ${\lambda}>{\kappa}$.
 Let $\<G_{\alpha}:{\alpha}<\cf({\lambda})\>$ be a good cut of $\mc A$ by (b). 
  
For each ${\alpha}<{\kappa}$ consider the family 
\begin{displaymath} 
    \mc A_{\alpha}=
    \{A\setm G_{\alpha}: A\in \mc A, A\subs G_{{\alpha}+1}, 
    |A\cap G_{\alpha}|<|A|\}.
    \end{displaymath}
Then  $\mc A_{\alpha}\ll \mc A$ and so   $\mc A_{\alpha}\in \mbb A$ by (a).

Since $|\bigcup \mc A_{\alpha}|\le |G_{{\alpha}+1}|<{\lambda}$,
by the inductive assumption, the family  
$\mc A_{\alpha}$ has a maximizing well order $\preceq_{\alpha}$.
We can assume that $\preceq_{\alpha}$ is a well-order of 
$G_{{\alpha}+1}\setm G_{\alpha}$.
Define the well ordering $\preceq$ of $\bigcup \mc A$ as follows:
\begin{enumerate}[(1)]
\item if ${\alpha}<{\beta}<{\lambda}$, then 
\begin{displaymath}
(G_{{\alpha}+1}\setm G_{\alpha})\prec (G_{{\beta}+1}\setm G_{\beta}),
\end{displaymath}
\item $\preceq\restriction (G_{{\alpha}+1}\setm G_{\alpha})=\preceq_{\alpha}$.
\end{enumerate}
Then $\preceq$ is a maximizing well-order of $\mc A$.
\end{proof}

\begin{proof}[Proof of Theorem \ref{tm:maxwo-gch}]

    We want to apply Theorem \ref{tm:gen-step-up}(2).
    So let  
    \begin{displaymath}
    \mbb A=\{\mc A:\exists {\lambda}\ge {\kappa}\ge {\omega}_2\ (\mc A\subs 
    {[{\lambda}]}^{{\kappa}}\text{ has property $\cad{2}{{\omega}}$})\}.
    \end{displaymath}
    
We should verify properties \ref{tm:gen-step-up}(a,b,c')

(a) is trivial by definition. 

To show (b) assume that 
$\mc A\in \mbb A$, ${\lambda}=\bigcup \mc A>{\kappa}$.
Next we should recall a lemma from \cite{HJSSz}:

\begin{lemma}[{\cite[Lemma 8.5]{HJSSz}}]\label{lm:gen_small_close}
Assume that $\lambda \ge \omega_2$ and ${\mu}^{\omega}={\mu}^+$
holds for each ${\mu}<{\lambda}$ with $\cf({\mu})={\omega}$. If
$\acal$ is an  hypergraph with property $\cad{2}{{\omega}}$ and $X$ is any
set with $\, |X| < \lambda$, then
\begin{displaymath}
 \big|\{A\in \acal:|X\cap A| > \omega \}\big| \le |X|.
\end{displaymath}
\end{lemma}

Next, by transfinite recursion  on ${\zeta}<{\lambda}$,   
define a a continuous sequence 
of subsets of ${\lambda}$,  $\<M_{\zeta}:{\zeta}<{\lambda}\>$, 
such that 
\begin{enumerate}[(1)]
\item $M_0=\empt $ and   
$
{\zeta}\subs M_{\zeta}\in {[{\lambda}]}^{{\kappa}+|{\zeta}|}$ for ${\zeta}>0$,
\item if $A\in \mc A$ and $|A\cap M_{\zeta}|>{\omega}$, then 
$A\subs M_{{\zeta}+1}$.
\end{enumerate}
By Lemma \ref{lm:gen_small_close}, if $M_{\zeta}$ is given, we can  choose 
a suitable $M_{{\zeta}+1}$.

    Let $\<{\lambda}_{\alpha}:{\alpha}<cf({\lambda})\>$ be a strictly increasing 
    cofinal sequence of limit ordinals in ${\lambda}$.
   
Let $G_{\alpha}=M_{{\lambda}_{\alpha}}$ for ${\alpha}<cf({\lambda})$ of $\mc A$.
We claim that $\<G_{\alpha}:{\alpha}<cf({\lambda})\>$ is a good cut.   

Let $A\in \mc A$ be arbitrary. Since $|A|\ge {\omega}_2$, 
there is ${\zeta}<{\lambda}$ such that $|A\cap M_{\zeta}|>{\omega}$,
and so $A\subs M_{{\zeta}+1}$.
Thus we can define 
\begin{displaymath}
{\alpha}=\min\{{\alpha}'<{\lambda}:|G_{{\alpha}'}\cap A|\ge {\omega}_2\}.
\end{displaymath}
Since $G_{\alpha}=M_{{\lambda}_{\alpha}}$ and 
${\lambda}_{\alpha}$ is a limit ordinal, there is ${\xi}<{\lambda}_{\alpha}$ with 
$|A\cap M_{\xi}|>{\omega}$, and so $A\subs M_{{\xi}+1}\subs G_{\alpha}$.

If ${\alpha}$ is a limit ordinal then there is ${\beta}<{\alpha}$
with ${\xi}<{\lambda}_{\beta}$, and so $A\subs M_{{\xi}+1}\subs G_{\beta}$ which
contradicts the minimality of ${\alpha}$.

Thus ${\alpha}={\beta}+1$ for some ${\beta}$, and so $A\subs G_{{\beta}+1}$
and $|G_{\beta}\cap A|\le {\omega}_1<{\kappa}$. 

Se we checked property \ref{tm:gen-step-up}(b).

\smallskip
 
(c) is straightforward:
if $\mc A\subs {[{\kappa}]}^{{\kappa}}$ is $\cad{2}{\omega}$,
then $|\mc A|\le {\kappa}$ by Lemma \ref{lm:gen_small_close}, so we can write 
$\mc A=\{A_{\zeta}:{\zeta}<{\xi}\}$.
Now let $\preceq$ is a well ordering of ${\kappa}$
such that writing $A_{\zeta}'=A_{\zeta}\setm A_{<{\zeta}}$ we have 
\begin{enumerate}[(1)]
\item $A'_{\xi}\prec A'_{\eta}$ for ${\xi}<{\eta}<{\lambda}$,
\item the order type of $A'_{\eta}$ is ${\kappa}+1$, i.e. 
$A'_{\eta}$ has a $\preceq$-maximal element.
\end{enumerate} 
Then $\preceq$ maximizes $\mc A$. 

Thus we verified condition \ref{tm:gen-step-up}(2)(b) and so 
we can apply Theorem \ref{tm:gen-step-up}(2) to show that every 
$\mc A\in \mbb A$ has a maximizing well-ordering.
\end{proof}

\begin{proof}[Proof of Theorem \ref{tm:maxwo-bo}]
    We want to apply Theorem \ref{tm:gen-step-up}(2).
    So let  
    \begin{displaymath}
    \mbb A=\{\mc A:\exists {\lambda}\ge {\kappa}\ge \bo >{\rho}\ (\mc A\subs 
    {[{\lambda}]}^{{\kappa}}\text{ has property $\cad{2}{{\rho}}$})\}.
    \end{displaymath}
    
We should verify properties \ref{tm:gen-step-up}(a,b,c').

(a) is trivial by definition.

To show (b) assume that 
$\mc A\in \mbb A$, ${\lambda}=\bigcup \mc A>{\kappa}$.
Next we should recall a lemma from \cite{So} which based on 
the following celebrated result of Shelah.

\newtheorem{shtheorem}{Shelah's Revised CGH Theorem\makebox[-2mm]{}} 
\renewcommand{\theshtheorem}{}

\begin{shtheorem}[{\cite[Theorem 0.1]{Sh-GCH}}]
    \label{tm:sh}
    \mbox{} \\If ${\mu}\ge \bo$, then ${\mu}^{[{\nu}]}={\mu}$
    for each large enough regular cardinal ${\nu}<\bo$.  
    \end{shtheorem}

\begin{lemma}[{\cite[Lemma 3.3]{So}}]\label{lm:loc_small}
    If ${\lambda}\ge  \bo >{\mu}$, and 
    $\{A_{\alpha}:{\alpha}<\tau\}\subs \br {\lambda};\bo;$
    is a hypergraph which property $\cad{2}{{\mu}}$, then $\tau\le {\lambda}$, and  
     there is an increasing, continuous sequence 
    $\<G_{\zeta}:{\zeta}<\cf({\lambda})\>\subs \br {\lambda};<{\lambda};$
     such that 
    \begin{equation}\label{eq:loc_small}
    \text{
    $\forall {\zeta}<\cf({\lambda})$\
    $\forall {\alpha}\in G_{\zeta+1}\setm G_{\zeta}$\ $(\ |A_{\alpha}\cap G_{\zeta}|<\bo$
    and $A_{\alpha}\subs G_{{\zeta}+1}\ )$.}
    \end{equation}
    \end{lemma}

    This Lemma just claims Claims \ref{tm:gen-step-up}(b).

    (c) is straightforward:
    if $\mc A\subs {[{\kappa}]}^{{\kappa}}$ is $\cad{2}{\omega}$,
    then $|\mc A|\le {\kappa}$ by Lemma \ref{lm:loc_small},  so we can write 
    $\mc A=\{A_{\zeta}:{\zeta}<{\xi}\}$. Now repeat an argument from the proof of the previous theorem: 
     let $\preceq$ is a well ordering of ${\kappa}$
    such that writing $A_{\zeta}'=A_{\zeta}\setm A_{<{\zeta}}$ we have 
    \begin{enumerate}[(1)]
    \item $A'_{\xi}\prec A'_{\eta}$ for ${\xi}<{\eta}<{\lambda}$
    \item the order type of $A'_{\eta}$ is ${\kappa}+1$, i.e. 
    $A'_{\eta}$ has a $\preceq$-maximal element.
    \end{enumerate} 
    Then $\preceq$ maximizes $\mc A$.
    
    Thus we verified condition \ref{tm:gen-step-up}(2)(b) and so 
    we can apply Theorem \ref{tm:gen-step-up}(2) to show that every 
    $\mc A\in \mbb A$ has a maximizing well-ordering.

\end{proof}

\section{Inhomogeneous families}\label{sc:inhomogeneous}

In \cite[Theorem 6]{Ko} Komjáth proved that $\MWO{{\lambda}}{{\kappa}}{2}{k}$, and so 
$\MVC{{\lambda}}{{\omega}\cup \{{\kappa}\}}{2}{k}$ 
 for each 
${\omega}\le {\kappa}\le {\lambda}$.

In this section we want to prove Theorem \ref{tm:nowhere-stac} which 
relaxes the requirements concerning the edge cardinality spectrum 
of the hypergraphs.

\begin{proof}[Proof of Theorem \ref{tm:nowhere-stac}]
Let $r$ be a fixed natural number. 

For   any infinite cardinal  ${\kappa}$ let $\prop{\kappa}$ be the following statement:
\begin{itemize}
    \item[] {\em   
    if 
    $S$ is a nowhere stationary set of cardinals with $\sup S\le {\kappa}$
    then every hypergraph  $\mc A\subs {[{\kappa}]}^{S}$ possessing  
    property  $C(2,r)$ has a minimal vertex cover.} 
\end{itemize}

We prove $\prop{\kappa}$ for each infinite ${\kappa}$ 
by transfinite induction on ${\kappa}$.

\smallskip

For  ${\kappa}={\omega}$, 
$\mc A\subs {[{\omega}]}^{<{\omega}}\cup {[{\omega}]}^{\omega}$, so 
$\prop {\omega}$  is just a special case of 
\cite[Corollary 9]{Ko}. 

\smallskip

So assume that ${\kappa}>{\omega}$ and $\prop{\mu}$ holds for 
each infinite cardinal ${\mu}<{\kappa}$.

\smallskip

To prove $\prop {\kappa}$ 
let  $S$ be a nowhere stationary set of cardinals with $\sup S\le {\kappa}$, and let  
$\mc A\subs {[{\kappa}]}^{S}$ be a hypergraph with property $\cad{2}{r}$.

Let  $\<{\kappa}_{\alpha}:{\alpha}<{\delta}\>$.
be the strictly increasing enumeration of $S\cap {\kappa}$.

Write $\mc B=\mc A\cap {[{\kappa}]}^{<{\kappa}}$
and $\mc C=\mc A\cap {[{\kappa}]}^{{\kappa}}$.
Adding pairwise disjoint "dummy" sets to $\mc A$ we can assume that $|\mc C|={\kappa}$.
Let 
$\<C_{\xi}:{\xi}<{\kappa}\>$  be an   
enumeration of $\mc C$.

\noindent{\bf Step 1.} {\em Construction of a chain of cardinals.}

We define an increasing sequence of cardinals $\<{\mu}_{\zeta}:{\zeta}<{\kappa}\>$
 as follows.

Let ${\mu}_0=0$.

If ${\kappa}={\mu}^+$ then let ${\mu}_{\zeta}={\mu}$ for each $1\le {\zeta}<{\kappa}$.

If ${\kappa}$ is an uncountable limit cardinal, 
let ${\rho}=cf({\delta}$.

Since $S$ is not stationary in ${\kappa}$, there is a 
a strictly increasing continuous sequence of infinite cardinals,  $\<{\nu}_{\alpha}:{\alpha}<{\rho}\>$,  which is cofinal 
in ${\kappa}$, and ${\nu}_{\alpha}\notin S$ for each limit ordinal ${\alpha}<{\rho}$. 
For $1\le {\zeta}<{\kappa}$ let $${\mu}_{\zeta}=\min\{{\nu}_{\alpha}:{\nu}_{\alpha}\ge {\zeta}\}.$$

Observe that the sequence $\<{\mu}_{\zeta}:{\zeta}<{\kappa}\>$ is increasing and continuous.

\noindent{\bf Step 2.} {\em The inductive construction.}

By transfinite recursion 
on  ${\zeta}<{\kappa}$ we  
will define  an increasing, continuous sequence $\<M_{\zeta}:{\zeta}<{\kappa}\>$ 
of subsets of ${\kappa}$, a sequence $\<Y_{\zeta}:{\zeta}<{\omega}\>$ of pairwise disjoint
subsets of ${\kappa}$, and a sequence $\<w_{\zeta}:{\zeta}<{\kappa}\>$ of functions
such that   
\begin{enumerate}[(1)\lsub{\zeta}]
    \item \label{first} \label{m-closed} ${\zeta}\subs M_{\zeta}\in {[{\kappa}]}^{{\mu}_{\zeta}}$,    
    \item \label{m-back} if $B\in \mc B $, $|B\cap M_{\zeta}|\ge r$ and $|B|\le |M_{\zeta}|$ 
    then 
    $B\subs M_{\zeta}$,
\item \label{Ywlimit} If ${\zeta}=0$ or ${\zeta}$ is a limit ordinal, then 
$M_{\zeta}=M_{<{\zeta}}$ and $Y_{\zeta}=w_{\zeta}=\empt$,
    \item \label{Yw} if ${\zeta}={\eta}+1$, then 
    $Y_{\zeta}\subs (M_{{\zeta}}\setm M_{\eta})\cap {\kappa}$
    and 
    $$w_{\zeta}:Y_{\zeta}\to 
    \big(\mc B[-Y_{\le {\eta}}]\sqcap M_{\zeta}\big)
    \cup\{C_{\eta}\},$$
\item\label{ay} 
$\met{\mc A}{Y_{\le {\zeta}}}\supset 
(\mc B\sqcap \mc M_{\zeta})\cup\{C_{\xi}:{\xi}<{\zeta}\}$,\smallskip
\item \label{wit} $w_{\le {\zeta}}(y)\cap Y_{\le {\zeta}}=\{y\}$ for each $y\in Y_{\le {\zeta}}$, \smallskip
\item \label{cinside}\label{last} if $C_{\eta}= w_{\zeta}(y)$ for some $y\in Y_{\zeta}$, then  
\begin{displaymath}
    C_{\eta}\cap \bigcup ((\nmet{\mc B}{Y_{\le {\eta}}})\cap {[{\kappa}]}^{\le |M_{\zeta}|})\subs 
    M_{\zeta}.
    \end{displaymath}
\end{enumerate}

\noindent{\bf Step 3.} {\em The inductive step.}

Assume that we have constructed $\<Y_{\xi}:{\xi}<{\zeta}\>$ and $\<w_{\xi}:{\xi}<{\zeta}\>$
such that \eqref{first}\lsub{{\xi}}-\eqref{last}\lsub{\xi} hold for ${\xi}<{\zeta}$.

\medskip 
If ${\zeta}=0$, then let $M_{\zeta}=Y_{\zeta}=w_{\zeta}=\empt$.
Since ${\mu}_0=0$, \eqref{m-closed}\lsub{{\zeta}} holds. The other requirements are trivial.

\medskip 
If ${\zeta}$ is a limit ordinal, take $M_{\zeta}=M_{<{\zeta}}$ 
and $Y_{\zeta}=w_{\zeta}=\empt$.  
\medskip 

Then \eqref{m-closed}\lsub{\zeta} holds because 
${\mu}_{\zeta}=\sup\{{\mu}_{\xi}:{\xi}<{\zeta}\}$.
\medskip 

To check \eqref{m-back}\lsub{\zeta}, assume that 
$B\in \mc B $,  $|B\cap M_{\zeta}|\ge r$ and $|B|\le |M_{\zeta}|$.

Since $M_{\zeta}=M_{<{\zeta}}$, there is ${\xi}<{\zeta}$ such that 
$|B\cap M_{\xi}|\ge r$.

Next we show that $|B|\le |M_{\eta}|$ for some ${\eta}<{\zeta}$.
        We know that $|B|\le |M_{\zeta}|$. If $|M_{\eta}|<|M_{\zeta}|$
        for each ${\eta}<{\zeta}$, then $|M_{\zeta}|={\mu}_{\zeta}={\nu}_{\alpha}$ for some limit ${\alpha}$, 
        and so $|M_{\zeta}|\notin S$, and so $|B|<|M_{\zeta}|$.
        Since $M_{\zeta}=\bigcup_{{\eta}<{\zeta}}M_{\eta}$, there is ${\eta}<{\zeta}$
        such that $|B|\le |M_{\eta}|$.

        Let ${\sigma}=\max({\xi},{\eta})$.

Then $|B\cap M_{\sigma}|\ge r$ and $|B|\le |M_{\sigma}|$, and so 
$B\subs M_{\sigma}\subs M_{\zeta}$ by \eqref{m-back}\lsub{\sigma}.
Thus \eqref{m-back}\lsub{\zeta} holds.

\medskip 

\eqref{Ywlimit}\lsub{\zeta} holds by the construction.

\medskip 

\eqref{Yw}\lsub{\zeta} is void. 

\medskip

To check \eqref{ay}\lsub{{\zeta}} first observe that for each ${\eta}<{\xi}$,
$C_{\eta}\in \met{\mc A}{Y_{\le {\eta}+1}}$
by \eqref{ay}\lsub{{{\eta}+1}}, so $\met{\mc A}{Y_{\le {\zeta}}}\supset 
\{C_{\xi}:{\xi}<{\zeta}\}$.

Assume that $B\in \mc B\sqcap M_{\zeta}$, i.e. $B\subs M_{\zeta}$ and $B\in \mc B$.
If  $B$ is finite then $B\subs M_{\sigma}$ for some ${\sigma}<{\zeta}$.
Thus $B\in \mc B\sqcap M_{\sigma}\subs \met{\mc A}{Y_{\le {\sigma}}}\subs \met{\mc A}{Y_{\le {\zeta}}}$

So we can assume that $B$ is infinite. 
Since $M_{\zeta}=M_{<{\zeta}}$, there is ${\xi}<{\zeta}$ such that 
$|B\cap M_{\xi}|\ge r$.

Next we show that $|B|\le |M_{\eta}|$ for some ${\eta}<{\zeta}$.
        We know that $|B|\le |M_{\zeta}|$. If $|M_{\eta}|<|M_{\zeta}|$
        for each ${\eta}<{\zeta}$, then $|M_{\zeta}|={\mu}_{\zeta}={\nu}_{\alpha}$ for some limit ${\alpha}$, 
        and so $|M_{\zeta}|\notin S$, and so $|B|<|M_{\zeta}|$.
        Since $M_{\zeta}=\bigcup_{{\eta}<{\zeta}}M_{\eta}$, there is ${\eta}<{\zeta}$
        such that $|B|\le |M_{\eta}|$.

        Let ${\sigma}=\max({\xi},{\eta})$.

Then $|B\cap M_{\sigma}|\ge r$ and $|B|\le |M_{\sigma}|$ and so 
$B\subs M_{\sigma}\subs M_{\zeta}$ by \eqref{m-back}\lsub{\sigma}.
Thus \eqref{ay}\lsub{\zeta} holds.

\medskip 

\eqref{wit}\lsub{\zeta}, 
and 
\eqref{cinside}\lsub{\zeta} are clear for limit ${\zeta}$
because $Y_{\le {\zeta}}=Y_{< {\zeta}}$
and $w_{\le {\zeta}}=w_{< {\zeta}}$.

\medskip

So for limit ${\zeta}$ we can carry out the inductive step.

\medskip 

Assume finally that ${\zeta}={\eta}+1.$

\smallskip

First we vertex cover $C_{\eta}$. 

If $C_{\eta}\in \met{\mc C}{Y_{\le\eta}}$, i.e. $C_{\eta}\cap Y_{\le {\eta}}\ne \empt$,  let 
\begin{displaymath}
\text{$Y'_{\zeta}=\{\}$ and $
v'_{\zeta}=\empt$.
}
\end{displaymath}

Assume that $C_{\eta}\in \nmet{\mc C}{Y_{\le\eta}
}$, i.e. $C_{\eta}\cap Y_{\le {\eta}}=\empt$.
Write $$C'_{\eta}=C_{\eta}\cap \bigcup((\nmet{\mc B}{Y_{\le {\eta}}})\cap 
{[{\kappa}]}^{\le {{\mu}_{\zeta}}}).$$
\smallskip
If $|C'_{\eta}|\le {\mu}_{\zeta}$, then 
$C''_{\eta}=((C_{\eta}\setm C'_{\eta})\setm \bigcup_{{\xi}<{\eta}}C_{\xi})$ has cardinality ${\kappa}$,
so we can put 
$$y_{\eta}=\min (C''_{\eta} \setm M_{\eta}).$$
In this case let 
\begin{displaymath}
\text{$Y'_{\zeta}=\{y_{\eta}\}$ and $v'_{\zeta}=\{\<y_{\eta},C_{\eta}\>\}$.
}
\end{displaymath}
If  $|C'_{\zeta}|>{\mu}_{\zeta}$, then  
$|(C'_{\zeta}\setm \bigcup_{{\xi}<{\eta}}C_{\xi})|>{\mu}_{\zeta}$ because 
$|(C_{\zeta}\cap \bigcup_{{\xi}<{\eta}}C_{\xi})|\le r|{\eta}|\le r|{\zeta}|\le {\mu}_{\zeta}$, so we can consider 
$$y_{\eta}=\min ((C'_{\zeta}\setm \bigcup_{{\xi}<{\eta}}C_{\xi})\setm M_{\eta})),
$$
and we can pick $B_{\eta}$ with   
$$y_{\eta}\in B_{\eta}\in \mc B[-Y_{\le {\eta}}]\cap {[{\kappa}]}^{\le {|M_{\zeta}|}}.$$
In this case let 
\begin{displaymath}
\text{$Y'_{\zeta}=\{y_{\eta}\}$ and $v'_{\zeta}=\{\<y_{\eta},B_{\eta}\>\}$.}
\end{displaymath}

\medskip

Next we want to define  sets  $M_{\zeta}$ and  $Z_{\zeta}$, and a function 
$v_{\zeta}$
such that 
$Y_{\zeta}=Y'_{\zeta}\cup Z_{\zeta}$ and  $w_{\zeta}=v'_{\zeta}\cup v_{\zeta}$
meet the requirements. 

If 
$C_{\eta}\in \met{\mc C}{Y_{\le\eta}}$, let 
$M_{\zeta}^-={\zeta}\cup M_{\eta}$.

If $C_{\eta}\in \nmet{\mc C}{Y_{\le\eta}}$ and $|C'_{\eta}|\le {\mu}_{\zeta}$
let $M_{\zeta}^-={\zeta}\cup M_{\eta}\cup Y'_{\eta}\cup C'_{\eta}$.

If $C_{\eta}\in \nmet{\mc C}{Y_{\le\eta}}$ and $|C'_{\eta}|> {\mu}_{\zeta}$
let $M_{\zeta}^-={\zeta}\cup M_{\eta}\cup Y'_{\eta}\cup B_{\eta}$.

Since $M'_{\zeta}\in {[{\kappa}]}^{{\mu}_{\zeta}}$,
using standard closing arguments, we can find a set  $M_{\zeta}\in {[{\kappa}]}^{{\mu}_{\zeta}}$ such that 
\begin{enumerate}[(a)]
\item $M_{\zeta}\supset M_{\zeta}'$,
\item if $B\in \mc B $, $|B\cap M_{\zeta}|\ge r$ and $|B|\le {\mu}_{\zeta}$, 
then 
$B\subs M_{\zeta}$.
\end{enumerate}

Write 
\begin{displaymath}
\mc C^w_{\le \eta}=\bigcup(\ran w_{\le {\eta}}\cap {[{\kappa}]}^{{\kappa}}).
\end{displaymath} 
and let $$\mc B_{\zeta}=\big(\nmet {(\mc B\sqcap M_{\zeta})}{(Y_{\le {\eta}}\cup Y'_{\zeta})}
\big )
$$

Consider the family 
$$\mc B_{\zeta}^-=\mc B_{\zeta}\restriction ((M_{{\zeta}}\setm M_{\eta})\setm (C_{\eta}\cup C^w_{\le {\eta}})).$$

Since $|\mc B^-_{\zeta}|\le |{[M_{\zeta}]}^{r}|={\mu}_{\zeta}$,  we can 
 apply the inductive assumption 
 $\prop{{\mu}_{\xi}}$
 for $\mc B_{\zeta}^-$ provided that we can show that 
 $\|\mc B_{\zeta}^-\|$ is nowhere stationary. 
The next lemma will actually yield that 
$\|\mc B_{\zeta}^-\|\subs  \|\mc B\|\cup{\omega}$.

\smallskip

\noindent{\bf Claim.} {\em $|C^w_{\le {\eta}}\cap B|<|B|$ for each 
$B\in (\nmet{\mc B}{Y_{\le {\eta}}})\cap  {[{\kappa}]}^{\le {{\mu}_{\zeta}}}$.
}

\begin{proof}[Proof of the Claim]
Fix $B\in \nmet{\mc B}{Y_{\le {\eta}}}\cap  {[{\kappa}]}^{\le {{\mu}_{\zeta}}}$.

Let ${\beta}=\sup\{{\gamma}:{\nu}_{\gamma}<|B|\}$. Since ${\nu}_{\varepsilon}\notin S$
for limit ${\varepsilon}$, we have ${\nu}_{\beta}<|B|$.

Assume that $w_{{\xi}+1}(y)=C_{\xi}$ and $B\cap (C_{\xi}\setm M_{\eta})\ne  \empt$.
Assume that ${\mu}_{\xi+1}={\nu}_{\varepsilon}$. Since the sequence $\<{\nu}_{\alpha}:{\alpha}<{\rho}\>$ is 
continuous, by the definition of ${\mu}_{{\xi}+1}$, ${\varepsilon}$ can not be a limit ordinal. So
 ${\mu}_{\xi+1}={\nu}_{\alpha+1}$ for some ${\alpha}$.

We know
\begin{displaymath}
C_{\xi}\cap \bigcup ((\nmet{\mc B}{Y_{\le {\xi}+1}})\cap {[{\kappa}]}^{\le {\mu}_{{\xi}+1}})\subs 
M_{{\xi}+1}.
\end{displaymath}
by \eqref{cinside}\lsub{{\xi}+1}, and so 
 $|B|>{\mu}_{{\xi}+1}= {\nu}_{{\alpha}+1}$.

Since ${\xi+1}\le{\mu}_{\xi+1}$ by the definition of ${\mu}_{\xi+1}$, we have 
${\xi}\in {\nu}_{{\alpha}+1}$.

Since ${\nu}_{{\alpha}+1}<|B|$,  it follows that ${\nu}_{{\alpha}+1}\le {\nu}_{\beta}$.
So 
$$\{{\xi}:C_{\xi}\in \ran w_{\le {\eta}}\land B\cap (C_{\xi}\setm M_{\eta})\ne  \empt\}
\subs {\nu}_{\beta}<|B|.$$

So we proved the Claim.
\end{proof}

If $B\in \mc B_{\zeta}$, then $B\notin \mc B\sqcap M_{\eta}$, so 
$|B\cap M_{\eta}|<r$ or $|B|>|M_{\eta}|$.
 Moreover $ B\subs M_{\zeta}$, 
so $|B\cap (M_{\zeta}\setm M_{\eta}) |=|B|$.
Since $|B\cap C^w_{\le {\eta}}|<|B|$ by the Claim and $|B\cap C_{\eta}|<r$,
we have $|B|=|B\cap  (M_{\zeta}\setm M_{\eta})\setm (C_{\eta}\cup C^w_{\le {\eta}})|$.
Thus $\|\mc B_{\zeta}^-\|\subs {\omega}\cup (K\cap {\kappa})$.

Since $\mc B^-_{\zeta}\subs \mc P(M_{\zeta})$ and $|M_{\zeta}|={\mu}_{\zeta}$ and 
$\|\mc B^-_{\zeta}\|$ is nowhere stationary, so we can apply the inductive assumption $\prop{{\mu}_{\zeta}}$
for $\mc B_{\zeta}^-$ to obtain a set  
$Z_{\zeta}\subs {\kappa}\cap (M_{{\zeta}}\setm M_{{\eta}})\setm (C_{\eta}\cup C^w_{\le {\eta}})$
which is a minimal vertex cover for $\mc B'_{\zeta}$,
and so it is a  a minimal vertex cover for $\mc B_{\zeta}$, Fix   a 
witnessing function $v_{\zeta}:Z_{\zeta}\to \mc B_{\zeta}$ 
such that 
$v_{\zeta}(y)\cap Z_{\zeta}=\{y\}$ 
for each $y\in Z_{\zeta}$.

Let 
\begin{displaymath}
Y_{\zeta}=\{y_{\zeta}\}\cup Y'_{\zeta}
\text{ and } w_{\zeta}=v'_{\zeta}\cup v_{\zeta}.
\end{displaymath}
Now we should check that \eqref{first}\lsub{\zeta}-\eqref{last}\lsub{\zeta} hold.

\eqref{m-closed}\lsub{\zeta} and  \eqref{m-back}\lsub{{\zeta}} hold by the choice of $M_{\zeta}$.

\eqref{Ywlimit}\lsub{\zeta} is void.

\eqref{Yw}\lsub{\zeta} is clear from the construction. 

    Next we check \eqref{ay}\lsub{\zeta}.     
    We have $\met{\mc A}{Y'_{\zeta}}\supset \mc B _{\zeta}$
    and $C_{\eta}\in \met{\mc A}{Y_{\le {\eta}}\cup Y'_{\eta}}$ by the construction, 
    and $\mc A[Y_{\le {\eta}}]\supset (\mc B\sqcap \mc M_{\eta})\cup\{C_{\xi}:{\xi}<{\eta}\}$
    by the inductive assumption   \eqref{ay}\lsub{\eta},
    so putting together we obtain 
    $$\mc A[Y_{\le {\zeta}}]\supset (\mc B\cap \mc M_{\zeta}\cap {[{\kappa}]}^{\le|M_{\zeta}|})\cup\{C_{\xi}:{\xi}<{\zeta}\}.$$
    So \eqref{ay}\lsub{\zeta} holds.
    
    Next we check \eqref{wit}\lsub{\zeta}. 
    Since $Y_{\zeta}\subs M_{\zeta}\setm M_{<{\zeta}}\setm C^w_{\le {\eta}}$, and 
     $\bigcup\ran w_{\le {\eta}}\subs M_{<{\zeta}}\cup  C^w_{\le {\eta}}$, 
    so
    \eqref{wit}\lsub{\zeta} holds for $y\in Y_{\le {\eta}}$. 
    
    If $y_{\eta}$ is defined, then $C_{\eta}\cap Y_{\le {\eta}}=\empt$,
    so $w_{\zeta}(y)\cap Y_{\zeta}=\{y_{\eta}\}$ because
    $Z_{\zeta}\cap C_{\eta}=\empt$.

    If $y\in Z_{\zeta}$, then $w_{\zeta}(y)\in \mc B[-(Y_{\le {\eta}}\cup Y'_{\zeta})]$, 
    and $w_{\zeta}(y)\cap Z_{\zeta}=\{y\}$
    so 
    $w_{\zeta}(y)\cap Y_{\le {\zeta}}=\{y\}$. 
    
    Thus \eqref{wit}\lsub{\zeta} holds. 
    
    Observe that \eqref{cinside}\lsub{\zeta} also holds:
    if $w_{\zeta}(y_{\eta})=C_{\eta}$, then $|C'_{\eta}|\le |M_{\zeta}|$, and  so 
    $C'_{\zeta}\subs M^-_{\zeta}$ implies $C'_{\zeta}\subs M_{\zeta}$.

\noindent{\bf Step 4.} {\em Conclusion.}

After the inductive construction $Y=\bigcup_{{\zeta}<{\kappa}}Y_{\zeta}$
is a minimal vertex cover of $\mc A$. The minimality is witnessed by the function 
$w=\bigcup_{{\zeta}<{\kappa}}w_{\zeta}$.
\end{proof}

\section{Hypergraphs possessing property $\cad{k}{r}$ for $k\ge 3$}
\label{sc:ckr}

Instead of Theorem \ref{tm:ckr} we will prove the following stronger theorem:
 
\begin{theorem}\label{tm:ckrmwo}
    $\MWO{{\lambda}}{{\omega}}{r}{k}$
    provided ${\omega}\le {\lambda}$ and $k,r\in {\omega}$
\end{theorem}
 
\begin{proof}[Proof of Theorem \ref{tm:ckrmwo}]

    We want to apply Theorem \ref{tm:gen-step-up}(2) to obtain the result.
    So let  
    \begin{displaymath}
    \mbb A=\{\mc A:\exists {\lambda}\ge  {\omega}_2\ (\mc A\subs 
    {[{\lambda}]}^{{\omega}}\text{ has property $\cad{k}{r}$})\}.
    \end{displaymath}
    
We should verify properties \ref{tm:gen-step-up}(a,b,c').

(a) is trivial by definition. 

(b) follows from the  Lemma \ref{lm:goodckr} below.

\begin{lemma}\label{lm:goodckr}
If  ${\omega}\le {\kappa}<{\lambda}$ are infinite cardinals, and  $\mc A\subs {[{\lambda}]}^{{\kappa}}$ has property $\cad{k}{r}$
for some $k,r\in {\omega}$, then $\mc A$ has a good cut.
\end{lemma}

\begin{proof}
For $x\in {[{\lambda}]}^{r}$ let
\begin{displaymath}
F(x)=\{A\in \mc A:x \subs A \}
\end{displaymath} 
Since $\mc A$ is $\cad{k}{r}$, $|F(x)|\le k-1$.

Let $\<{\lambda}_{\alpha}:{\alpha}<{\omega}\cdot\cf({\lambda})\>$
be an increasing continuous cofinal sequence in ${\lambda}$ with 
${\lambda}_0=0$.
Define an increasing, continuous  sequence
$\<H_{\eta}:{\eta}<{\omega}\cdot\cf({\lambda})\>$
such that 
\begin{enumerate}[(1)]
 \item $H_0=\empt$
 \item $H_{{\eta}+1}=H_{\eta}\cup ({\eta}+1)\cup\bigcup \{F(x):
 x\in {[H_{\eta}]}^{r}\}$.   
\end{enumerate} 
Since $F(x)\le k-1$, $|H_{\eta}|\le |{\eta}|+{\omega}$.

Moreover $|A\cap H_{\eta}|\ge r$ implies $A\subs H_{{\eta}+1}$.

Let $G_{\alpha}=H_{{\omega}\cdot {\alpha}}$.
Then $\<G_{\alpha}:{\alpha}<\cf({\lambda})\>$ is a good cut of $\mc A$. 
\end{proof}

Property  (c') follows from the next lemma.

\begin{lemma}\label{lm:ckr} 
    If $\mc A$ is a countable hypergraph which possesses property $P(k,r)$,
    then $\mc A$ has a maximizing well-ordering  $\prec $.
   \end{lemma}
   
   \begin{proof}[Proof of Lemma \ref{lm:ckr}]
    We prove it by induction on $k$. If $k=1$, then we have that for all 
    $n\in \omega $, $|A_{n}|\le r-1$, so $A_{n}$ is a finite set. Choose
     any well-ordering $\prec $ on $\bigcup_{n\in \omega}A_{n}$,  then clearly all non-empty finite sets have maximal element by $\prec $.
    
    Now suppose the lemma is true for hypergraphs possessing property $\cad{k}{r}$ and let $\mc A$ be a 
    hypergraph with property $\cad{k+1}{r}$, and write  
    $\mc A=\{A_{n}:n<\omega \}$. 
    For $n<{\omega}$ let $$B_{n}=A_{n}\setm \bigcup_{i<n}A_{i},$$ the disjointification of $A_{n}$. 
    Clearly $B_{n}\cap B_{m}=\emptyset $ if $n\neq m$ and $\bigcup_{n\in \omega}B_{n}=\bigcup_{n\in \omega}A_{n}$, but 
    $B_{n}=\emptyset $ is possible even is $A_{n}\neq \emptyset $. Choose any $n$, such that $B_{n}\neq \emptyset $. 

    We will show that the system $\{A_{m}\cap B_{n}:m>n,A_{m}\neq A_{n} \} )$ is $\cad{k}{r}$. Choose any $m_{1},\dots,m_{k}>n$, 
    such that the sets $A_{m_{i}}\cap B_{n}$ are pairwise different for $1\le i\le k$. The the sets $A_{m_{i}}$ are also pairwise different.
     For the intersection 
     \begin{multline*}
        (A_{m_{1}}\cap B_{n})\cap \dots\cap (A_{m_{k}}\cap B_{n})=A_{m_{1}}\cap \dots\cap A_{m_{k}}\cap B_{n}\subseteq\\ 
     A_{m_{1}}\cap \dots\cap A_{m_{k}}\cap A_{n}
        \end{multline*}
     and $A_{m_{1}},\dots,A_{m_{k}},A_{n}$ are pairwise different, 
     so $|(A_{m_{1}}\cap B_{n})\cap \dots \cap (A_{m_{k}}\cap B_{n})|\le |A_{m_{1}}\cap \dots \cap A_{m_{k}}\cap A_{n}|\le r-1$, as $\mc A$ has property  
     $\cad{k+1}{r}$.
    
    By the inductive assumption  there is a well ordering $\prec '_{n}$ on 
    $$\bigcup { \{A_{m}\cap B_{n}:m>n,A_{m}\neq A_{n} \} ) },$$ such that for all $m>n$,
     where $A_{m}\neq A_{n}$ and $A_{m}\cap B_{n}\neq \emptyset $, 
    then it has a maximal element. We can make a $\prec _{n}$ well-ordering from it on $B_{n}$, so
     that $B_{n}$ has a maximal element, and for all $m>n$, if $A_{m}\cap B_{n}\neq \emptyset $, 
     then it has a maximal element. First if $B_{n}$ has elements that are not in any $A_{m}\cap B_{n}$, where $m>n,A_{m}\neq B_{n}$, 
    then we well-order them into the bottom of $B_{n}$, so the well-ordering is defined on the whole $B_{n}$.
     If $B_{n}$ does not have a maximal element by $\prec '_{n}$, then we pick any $v\in B_{n}$, and 
     elevate it into the top of $B_{n}$,
      while remaining the order of other elements. Doing this the order $\prec _{n}$ clearly remains a well-order.
       Then $v$ will be the new maximum of $B_{n}$. 
     For $m>n$ if $A_{m}\neq A_{n}$ and $A_{m}\cap B_{n}\neq \emptyset $,  a maximal element by $\prec '_{n}$. 
     If $v\not\in A_{m}\cap B_{n}$, then its order does not change, so it keeps its maximal element. 
     If $v\in A_{m}\cap B_{n}$, then it will be its new maximal element, as it is even maximal in $B_{n}$. Finally if $A_{m}=A_{n}$, then $A_{m}\cap B_{n}=A_{n}\cap B_{n}=B_{n}$, that have a maximal element.
    
    Now we define the ordering $\prec $. For any 
    $u,v\in \bigcup_{n\in \omega}A_{n}=\bigcup_{n\in \omega}B_{n}$ there are $i,j\in \omega $, such that $u\in B_{i},v\in B_{j}$. 
    If $i<j$, then $u\prec v$, if $j<i$, then $v\prec u$, and if $i=j$, then we order them by $\prec _{i}$ on $B_{i}$.
     By other words we take the well orders $\prec _{n}$-s on the top of each other, and then $\prec $ is also a well-order. Choose any $n\in \omega $, that $A_{n}\neq \emptyset $. 
    If $B_{n}\neq \emptyset $, then it has a maximal element. Since all elements of $A_{n}$ are either in $B_{n}$ or in 
    some $B_{i}$, where $i<n$, we have $\max(A_{n})=\max(B_{n})$. If $B_{n}=\emptyset $, then we have $\emptyset \neq A_{n}=\bigcup _{i<n}{B_{i}}$. 
    Let $j<n$ be the maximal, 
    such that $A_{n}\cap B_{j}\neq \emptyset $. 
    Then by the construction of $\prec _{j}$, we have that $A_{n}\cap B_{j}$ has a maximal element. 
    Since all elements of $A_{n}$ are even in $A_{n}\cap B_{j}$ or in some $A_{n}\cap B_{i}$, where $i<j$, it is also maximal in $A_{n}$. Then the well ordering $\prec $ is good, so the induction step works from $k$ to $k+1$.
    \end{proof}
   
So we can apply Theorem \ref{tm:gen-step-up}(2) to derive that 
every $\mc A\in \mbb A$ has a maximizing well-order. 

\end{proof}

\begin{definition}\label{df:min}
If $\mc E$ is a hypergraph, let $\min \mc E$ be the family of $\subs$-minimal elements of 
$\mc E$.
\end{definition}
Observe that if $\mc E$ is $\cad{k}{r}$, then for each $E\in \mc E$
there is $E'\in \min \mc E$ with $E'\subs E$. So a minimal vertex cover of $\min \mc E$
will be a minimal vertex cover of $\mc E$. 

\begin{proof}[Proof of Theorem \ref{tm:c3k}]

We prove first that $\MVC{{\omega}_1}{{\omega}_1}{3}{r}$ for $r<{\omega}$.

Let $\mc D\subs {[{\omega}_1]}^{{\omega}_1}$ be a hypergraph having property $\cad{3}{r}$.

Write $\mc D=\{D_{\alpha}:{\alpha}<{\omega}_1\}$.
For each ${\alpha}<{\omega}_1$
pick 
\begin{displaymath}
y_{\alpha}\in D_{\alpha}\setm\{D_{\xi}\cap D_{\zeta}:{\xi}<{\zeta}<{\alpha}\}.
\end{displaymath}

Let $Y=\{y_{\alpha}:{\alpha}<{\omega}_1\}$.

Then $\mc D[Y]=\mc D$.  

Moreover,  $\mc D\restriction Y$ is $\cad{3}{r}$ and $\cad{2}{{\omega}_1}$
because $y_{\alpha}\notin D_{\zeta}\cap D_{\xi}$ for ${\alpha}>\max({\zeta},{\xi})$.

Let $\mc A=\min (\mc D\restriction Y)$.

Then a minimal vertex cover of $\mc A$ will be a minimal vertex cover of $\mc D$.

So to find a minimal vertex cover of $\mc D$ it is enough  to prove the following lemma:

\begin{lemma}\label{lm:ckrc2oo}
Assume that a  hypergraph  $\mc A\subs \mc P({\omega}_1)$
possesses the following properties: 
\begin{enumerate}[(A)]
\item \label{pr:c3kckr} $\cad{k}{r}$,
\item \label{pr:cKrc2oo} $\cad{2}{{\omega}_1}$,
\item \label{pr:ckrsymdiff}  $A\setm A'\ne \empt$ for each $\{A,A'\}\in {[\mc A]}^{2}$.
\end{enumerate}
Then $\mc A$ has a minimal vertex cover.
\end{lemma}

\begin{proof}[Proof of the Lemma \ref{lm:ckrc2oo}]
   
    Write $\mc B=\mc A\cap {[{\omega_1}]}^{\le{\omega}}$
    and $\mc C=\mc A\cap {[{\omega_1}]}^{{\omega_1}}$.
    Adding pairwise disjoint "dummy" sets to $\mc A$ we can assume that $|\mc C|={\omega_1}$.
    Let 
    $\<C_{\xi}:{\xi}<{\omega_1}\>$  be an   
    enumeration of $\mc C$.

    \noindent{\bf Step 1.} {\em The inductive construction.}

    By transfinite recursion 
    on  ${\zeta}<{\omega_1}$ we  
    will define  an increasing, continuous sequence $\<M_{\zeta}:{\zeta}<{\omega_1}\>$ 
    of subsets of ${\omega_1}$, a sequence $\<Y_{\zeta}:{\zeta}<{\omega}\>$ of pairwise disjoint
    subsets of ${\omega_1}$, and a sequence $\<w_{\zeta}:{\zeta}<{\omega_1}\>$ of functions
    such that   
    \begin{enumerate}[(1)\lsub{\zeta}]
        \item \label{2first} \label{2m-closed} ${\zeta}\subs M_{\zeta}\in {[{\omega_1}]}^{{\omega}}$
        for $1\le {\zeta}<{\omega}_1$, $M_0=\empt$,    
        \item \label{2m-back} if $B\in \mc B $ and $|B\cap M_{\zeta}|\ge r$ 
        then 
        $B\subs M_{\zeta}$,
    \item \label{2Ywlimit} If ${\zeta}=0$ or ${\zeta}$ is a limit ordinal, then 
    $M_{\zeta}=M_{<{\zeta}}$ and $Y_{\zeta}=w_{\zeta}=\empt$,
        \item \label{2Yw} if ${\zeta}={\eta}+1$, then 
        $Y_{\zeta}\subs (M_{{\zeta}}\setm M_{\eta})\cap {\omega_1}$
        and 
        $$w_{\zeta}:Y_{\zeta}\to 
        \big((\mc B\sqcap M_{\zeta})\setm (\mc B\sqcap M_{\eta})\big)
        \cup\{C_{\eta}\},$$
    \item\label{2ay} 
    $\met{\mc A}{Y_{\le {\zeta}}}\supset 
    (\mc B\sqcap \mc M_{\zeta})\cup\{C_{\xi}:{\xi}<{\zeta}\}$,\smallskip
    \item \label{2wit} $w_{\le {\zeta}}(y)\cap Y_{\le {\zeta}}=\{y\}$ for each $y\in Y_{\le {\zeta}}$, \smallskip
    \item \label{2cinside}\label{2last} if $C_{\eta}= w_{\zeta}(y)$ for some $y\in Y_{\zeta}$, then  
    \begin{displaymath}
        C_{\eta}\cap \bigcup (\nmet{\mc B}{Y_{\le {\eta}}})\subs 
        M_{\zeta}.
        \end{displaymath}
    \end{enumerate}
      
    \noindent{\bf Step 2.} {\em The inductive step.}

    Assume that we have constructed $\<Y_{\xi}:{\xi}<{\zeta}\>$ and $\<w_{\xi}:{\xi}<{\zeta}\>$
    such that \eqref{first}\lsub{{\xi}}--\eqref{last}\lsub{\xi} hold for ${\xi}<{\zeta}$.
    
    \medskip 
    If ${\zeta}=0$, then let $M_{\zeta}=Y_{\zeta}=w_{\zeta}=\empt$.
    All requirements are trivial.
    
    \medskip 
    If ${\zeta}$ is a limit ordinal, put $M_{\zeta}=M_{<{\zeta}}$ 
    and $Y_{\zeta}=w_{\zeta}=\empt$.  
    \medskip 
    
    Then \eqref{m-closed}\lsub{\zeta} is trivial.
    \medskip 
    
    To check \eqref{m-back}\lsub{\zeta}, assume that 
    $B\in \mc B $ and $|B\cap M_{\zeta}|\ge r$ .
    
    Since $M_{\zeta}=M_{<{\zeta}}$, there is ${\xi}<{\zeta}$ such that 
    $|B\cap M_{\xi}|\ge r$.
    So
    $B\subs M_{\xi}\subs M_{\zeta}$ by \eqref{m-back}\lsub{\xi}.
    Thus \eqref{m-back}\lsub{\zeta} holds.
    
    \medskip 
    
    \eqref{Ywlimit}\lsub{\zeta} holds by the construction.

    \medskip 
    
    \eqref{Yw}\lsub{\zeta} is void. 
    
    \medskip

    To check \eqref{ay}\lsub{{\zeta}} first observe that for each ${\eta}<{\xi}$,
    $C_{\eta}\in \met{\mc A}{Y_{\le {\eta}+1}}$
    by \eqref{ay}\lsub{{{\eta}+1}}, so $\met{\mc A}{Y_{\le {\zeta}}}\supset 
    \{C_{\xi}:{\xi}<{\zeta}\}$.
    
    Assume that $B\in \mc B\sqcap M_{\zeta}$, i.e. $B\subs M_{\zeta}$ and $B\in \mc B$.
    If  $B$ is finite then $B\subs M_{\sigma}$ for some ${\sigma}<{\zeta}$.
    Thus $B\in \mc B\sqcap M_{\sigma}\subs \met{\mc A}{Y_{\le {\sigma}}}\subs \met{\mc A}{Y_{\le {\zeta}}}$
    
    So we can assume that $B$ is infinite. 
    Since $M_{\zeta}=M_{<{\zeta}}$, there is ${\xi}<{\zeta}$ such that 
    $|B\cap M_{\xi}|\ge r$.

    Then $|B\cap M_{\xi}|\ge r$  so 
    $B\subs M_{\xi}\subs M_{\zeta}$ by \eqref{m-back}\lsub{\xi}.
    Thus \eqref{ay}\lsub{\zeta} holds.

    \medskip 
    
    \eqref{wit}\lsub{\zeta}, 
    and 
    \eqref{cinside}\lsub{\zeta} are clear for limit ${\zeta}$
    because $Y_{\le {\zeta}}=Y_{< {\zeta}}$
    and $w_{\le {\zeta}}=w_{< {\zeta}}$.

    \medskip

    So for limit ${\zeta}$ we can carry out the inductive step.

    \medskip 
    
    Assume finally that ${\zeta}={\eta}+1.$
    
    \smallskip
    
    First we vertex cover $C_{\eta}$. 

    \smallskip

    \noindent{\bf Case 1.} $C_{\eta}\in \met{\mc C}{Y_{\le\eta}}$, i.e. $C_{\eta}\cap Y_{\le {\eta}}\ne \empt$.

    Let 
    \begin{displaymath}
    \text{$Y'_{\zeta}=\{\}$ and $
    v'_{\zeta}=\empt$
    }
    \end{displaymath}
and 
    \begin{displaymath}
        M_{\zeta}^-={\zeta}\cup M_{\eta}.
    \end{displaymath}

    \noindent{\bf Case 2.}
    Assume that $C_{\eta}\in \nmet{\mc C}{Y_{\le\eta}
    }$, i.e. $C_{\eta}\cap Y_{\le {\eta}}=\empt$.

Let \begin{displaymath}
\mc B'_{\zeta}=\min (\mc B[-Y_{\le {\eta}}]\restriction ({\omega}_1\setm M_{\eta}))
\end{displaymath}
Observe that if $B\in \mc B[-Y_{\le {\eta}}]$, then 
$B\setm M_{\eta}\ne \empt$ by \eqref{2ay}\lsub{{\eta}} .
    Write $$C'_{\eta}=C_{\eta}\cap \bigcup\mc B_{\zeta}'.$$

    \noindent{\bf Case 2.1.} $C'_{\eta}=\empt$.

    Since 
    $C''_{\eta}=C_{\eta}\setm \bigcup_{{\xi}<{\eta}}C_{\xi}$ has cardinality 
    ${\omega}_1$,
    so we can define 
    $$y_{\eta}=\min (C''_{\eta} \setm M_{\eta}).$$
    In this case let 
    \begin{displaymath}
    \text{$Y'_{\zeta}=\{y_{\eta}\}$ and $v'_{\zeta}=\{\<y_{\eta},C_{\eta}\>\}$
    }
    \end{displaymath}
and  
    \begin{displaymath}
        M_{\zeta}^-={\zeta}\cup M_{\eta}\cup Y'_{\zeta}.
    \end{displaymath}

    \noindent{\bf Case 2.2.} $C'_{\zeta}\ne \empt$.

    Then
 put
    $$y_{\eta}=\min (C'_{\zeta})\in {\omega}_1\setm M_{\eta},
    $$
    and  pick $B_{\eta}\in B[-Y_{\le {\eta}}]$ with   
    $$y_{\eta}\in (B_{\eta}\setm M_{\eta})\in \mc B'_{\zeta}.$$
    In this case let 
    \begin{displaymath}
    \text{$Y'_{\zeta}=\{y_{\eta}\}$ and $v'_{\zeta}=\{\<y_{\eta},B_{\eta}\>\}$}
    \end{displaymath}
and  
    \begin{displaymath}
        M_{\zeta}^-={\zeta}\cup M_{\eta}\cup Y'_{\zeta}\cup B_{\eta}.
    \end{displaymath}

    \medskip
    
    Next we want to define  sets  $M_{\zeta}$ and  $Z_{\zeta}$, and a function 
    $v_{\zeta}$
    such that 
    $Y_{\zeta}=Y'_{\zeta}\cup Z_{\zeta}$ and  $w_{\zeta}=w'_{\zeta}\cup v_{\zeta}$
    meet the requirements.

    Using standard closing arguments we can find a set  $M_{\zeta}\in {[{\omega_1}]}^{\omega}$ such that 
    \begin{enumerate}[(a)]
    \item $M_{\zeta}\supset M_{\zeta}'$,
    \item if $B\in \mc B $ and $|B\cap M_{\zeta}|\ge r$ 
    then 
    $B\subs M_{\zeta}$.
    \end{enumerate}

If we are in Case 1 or in Case 2.1, then let 
\begin{displaymath}
    \mc B_{\zeta}=\nmet {(\mc B\sqcap M_{\zeta})}{(Y_{\le {\eta}})}
    \restriction \big({\omega}_1\setm M_{\eta}\big ).
\end{displaymath}
Then $\empt\notin \mc B_{\zeta}\subs \mc P(M_{\zeta})$ and $\mc B_{\zeta}$ possesses property
$\cad{3}{k}$, so 
by theorem Theorem \ref{tm:ckr} the hypergraph $\mc B_{\zeta}$ has a minimal vertex cover $Z_{\zeta}$
and a witnessing function $v_{\zeta}$.

Assume that we are in Case 2.2. 

Let 
\begin{displaymath}
\mc B_{\zeta}=\{(B\setm M_{\eta})\setm B_{\eta}:B\in 
\nmet 
{(\mc B\sqcap M_{\zeta})}
{(Y_{\le {\eta}}\cup Y'_{\zeta}})
\}.
\end{displaymath}
Since $B_{\eta}\setm M_{\eta}$ was a minimal element of 
$\mc B[-Y_{\le {\eta}}]\restriction ({\omega}_1\setm M_{\eta})$, it follows
that $\empt\notin \mc B_{\eta}$.
So
by theorem Theorem \ref{tm:ckr} the hypergraph $\mc B_{\zeta}$ has a minimal vertex cover $Z_{\zeta}$
and a witnessing function $v_{\zeta}$.

    \smallskip

    Let 
    \begin{displaymath}
    Y_{\zeta}=Y'_{\zeta}\cup Z_{\zeta}
    \text{ and } w_{\zeta}=v'_{\zeta}\cup v_{\zeta}.
    \end{displaymath}
    Now we should check that \eqref{first}\lsub{\zeta}-\eqref{last}\lsub{\zeta} hold.

    \eqref{m-closed}\lsub{\zeta} and  \eqref{m-back}\lsub{{\zeta}} hold by the choice of $M_{\zeta}$.

    \eqref{Ywlimit}\lsub{\zeta} is void.
    
    \eqref{Yw}\lsub{\zeta} is clear from the construction.

    Next we check \eqref{ay}\lsub{\zeta}.     
    We have $\met{\mc A}{Y'_{\zeta}}\supset \mc B _{\zeta}$
    and $C_{\eta}\in \met{\mc A}{Y_{\le {\eta}}\cup Y'_{\eta}}$ by the construction, 
    and $\mc A[Y_{\le {\eta}}]\supset (\mc B\sqcap \mc M_{\eta})\cup\{C_{\xi}:{\xi}<{\eta}\}$
    by the inductive assumption   \eqref{ay}\lsub{\eta},
    so putting together we obtain 
    $\mc A[Y_{\le {\zeta}}]\supset (\mc B\cap \mc M_{\zeta})\cup\{C_{\xi}:{\xi}<{\zeta}\}$.
    So \eqref{ay}\lsub{\zeta} holds.
    
    Next we check \eqref{wit}\lsub{\zeta}.

    Write 
    \begin{displaymath}
    \mc C^w_{\le \eta}=\bigcup(\ran w_{\le {\eta}}\cap {[{\omega_1}]}^{{\omega_1}}),
    \end{displaymath} 
and observe that 
\begin{equation}\label{eq:CB}\tag{$\ddag $}
\text{$C^w_{\le {\eta}}\cap B=\empt$ for each 
$B\in \nmet{\mc B}{Y_{\le {\eta}}}$.}
\end{equation}

        Since $Y_{\zeta}\subs M_{\zeta}\setm M_{<{\zeta}}\setm C^w_{\le {\eta}}$, and 
         $\bigcup\ran w_{\le {\eta}}\subs M_{<{\zeta}}\cup  C^w_{\le {\eta}}$, 
        so
        \eqref{wit}\lsub{\zeta} holds for $y\in Y_{\le {\eta}}$. 
        
        If $y_{\eta}$ is defined, then $C_{\eta}\cap Y_{\le {\eta}}=\empt$,
        so $w_{\zeta}(y)\cap Y_{\zeta}=\{y_{\eta}\}$ because
        $Z_{\zeta}\cap C_{\eta}=\empt$.

        If $y\in Z_{\zeta}$, then $w_{\zeta}(y)\in \mc B[-(Y_{\le {\eta}}\cup Y'_{\zeta})]$, 
        and $w_{\zeta}(y)\cap Z_{\zeta}=\{y\}$
        so 
        $w_{\zeta}(y)\cap Y_{\le {\zeta}}=\{y\}$. 
        
        Thus \eqref{wit}\lsub{\zeta} holds. 
        
        Observe that \eqref{cinside}\lsub{\zeta} also holds:
        if $w_{\zeta}(y_{\eta})=C_{\eta}$, then $C'_{\eta}=\empt$, which is just the statement of 
        \eqref{cinside}\lsub{\zeta}.
    
    \noindent{\bf Step 4.} {\em Conclusion.}
    
    After the inductive construction $Y=\bigcup_{{\zeta}<{\kappa}}Y_{\zeta}$
    is a minimal vertex cover for $\mc A$. The minimality is witnessed by the function 
    $w=\bigcup_{{\zeta}<{\kappa}}w_{\zeta}$.
\end{proof}

\medskip

    We want to apply Theorem \ref{tm:gen-step-up}(1) to conclude the proof of Theorem  \ref{tm:c3k}.
    So let  
    \begin{displaymath}
    \mbb A=\{\mc A:\exists {\lambda}\ge  {\omega}\ (\mc A\subs 
    {[{\lambda}]}^{{\omega}_1}\text{ has property $\cad{3}{r}$})\}.
    \end{displaymath}

    Then \ref{tm:gen-step-up}(a) is trivial bt definition.
    \ref{tm:gen-step-up}(b) follows from Lemma  \ref{lm:goodckr}.
Finally \ref{tm:gen-step-up}(c) is just Lemma \ref{lm:ckrc2oo}.
\end{proof}

\section{Problems}

\begin{enumerate}[(1)]
\item  Does $\MVC{{\omega}_1}{{\omega}_1}{4}{r}$ hold?
\item  Does $\MVC{{\omega}_2}{{\omega}_2}{3}{r}$ hold?
\end{enumerate}

\end{document}